\documentclass[12pt]{article}
\usepackage{amsfonts}
\usepackage[latin1]{inputenc} \RequirePackage[T1]{fontenc}
\usepackage[english]{babel} \RequirePackage{amsmath}
\usepackage{amssymb} \RequirePackage{amsthm}
\usepackage{amsmath}
\usepackage{pslatex}
\usepackage{fancyhdr}
\usepackage{comment}
\usepackage{graphicx}
\usepackage{geometry}
\numberwithin{equation}{section}
\usepackage{color}

\newtheorem{theo}{Theorem}[section]

\newtheorem{rem}{Remark}[section]

\newtheorem{defi}{Definition}[section]

{\penalty-20\null\hfill$\square$\par\medbreak}

\newcommand{\R}        {{{\rm I\! R}}}

\newcommand{\E}        {{{\rm I\! E}}}

\renewcommand{\P}        {{ {\rm I \hskip -2pt P}}}

\begin{document}
\title{General Fully Coupled Forward Backward Stochastic Differential Equations with delayed generator}
	
	\author{Auguste Aman$^{a}$ \thanks{aman.auguste@ufhb.edu.ci/ augusteaman5@yahoo.fr}\;\; Harouna Coulibaly$^{b}$ \thanks{harounbolz@yahoo.fr}\; and\; Jasmina \DJ or\dj evi\'c $^{c,d}$ \thanks{jasmindj@math.uio.no/djordjevichristina@gmail.com, corresponding author} \\ 
a.	UFR Maths et Informatique, Université Félix H. Boigny, Abidjan, Côte d'Ivoire\;\;\;\;\;\;\;\;\;\;\;\;\;\;\;\;\\
b. Ecole Superieure Africaines des TIC, Abidjan, Côte d'Ivoire\;\;\;\;\;\;\;\;\;\;\;\;\;\;\;\;\;\;\;\;\;\;\;\;\;\;\;\;\;\;\;\;\;\;\;\;\;\;\;\;\;\;\;\;\\
c. Department of Mathematics, University of  Oslo, Norway\;\;\;\;\;\;\;\;\;\;\;\;\;\;\;\;\;\;\;\;\;\;\;\;\;\;\;\;\;\;\;\;\;\;\;\;\;\;\;\;\;\;\;\;\;\;\;\;\;\\
d. Faculty of Science and Mathematics, University of Ni\v s, Serbia\;\;\;\;\;\;\;\;\;\;\;\;\;\;\;\;\;\;\;\;\;\;\;\;\;\;\;\;\;\;\;\;\;\;\;\;\;\;\;\;\;}

	\date{}
	\maketitle \thispagestyle{empty} \setcounter{page}{1}
	
	% ------- [First Page Running Head] - place it immediately after title! ------
	\thispagestyle{fancy} \fancyhead{}
	\fancyfoot{}
	\renewcommand{\headrulewidth}{0pt}
	%------------------------------------
	\begin{abstract}
	This paper is devoted to study different type of BSDE with delayed generator. We first establish an existence and uniqueness result under delayed Lipschitz condition for non homogenous backward stochastic differential equation with delayed generator. Next, existence and uniqueness result for general fully coupled forward backward stochastic differential equation with delayed coefficients has been derived.
	\end{abstract}

	\vspace{.08in}\textbf{MSC}: 34F05; 60H10; 60F10; 60H30 \\
	\vspace{.08in}\textbf{Keywords}: forward-backward Stochastic differential equations; delayed generators
	
	\section{Introduction}
	
	Backward stochastic differential equations (shorter BSDEs) were developed in the early 1990s, by Pardoux and Peng  \cite{pp}. They established  results on the existence and uniqueness of the adapted solutions under Lipschitz condition in the mentioned paper. Since then, BSDEs have been intensively developed both theoretically and in various applications. In  papers \cite{pp2} by Pardoux and Peng and \cite{pp3} by Pardoux, authors gave a probabilistic representation for the solutions of some quasilinear parabolic partial differential equations in terms of solutions of BSDEs, by obtaining a generalization of the well-known Feynman-Kac formula. Furthermore, BSDEs are encountered in many fields of applied mathematics such as finance, stochastic games and optimal control, as well as partial differential equations and homogenization etc (\cite {el1}, \cite{el2},\cite{el3},\cite{ham}).
	
	 There exists an extensive literature on BSDEs, depending on their direction of generalization and their applications. Some authors generalized noise but introducing another independent Brownian motion  in the equation (named backward doubly stochastic differential equations), Lévy processes, further more martingales, processes with memory etc. Each of this generalizations induced some new fields of applications. Those equations which have time dependent coefficients are specially interesting because of their wide application.  We will restrict our selfs to the case of forward-backward equations (shorter FBSDEs) with a special delay generator type which we will describe in the sequel.
	 
	 \medskip
	 
	 In \cite{DI}  Delong and Imkeller introduced backward stochastic differential equations with time delayed generators, where generator at time $t$ can depend on the values of a solution in the past, weighted with a time delay function for instance of the moving average type. They proved existence and uniqueness of a solution for a sufficiently small time horizon or for a sufficiently small Lipschitz constant of a generator, and gave examples of BSDE with time delayed generators that have multiple solutions or that have no solutions. Furthermore, the same authors in \cite{DI1} extended their analysis on  BSDEs with time delayed generators driven by Brownian motions and Poisson random measures, that constitute the two components of a L\'evy process.  In this case also generator can depend on the past values of a solution, by feeding them back into the dynamics with a time lag. For such time delayed BSDEs, they proved the existence and uniqueness of solutions provided they restricted the case on a sufficiently small time horizon or the generator possesses a sufficiently small Lipschitz constant.
	 
	The difference should be made between forward and backward stochastic differential equations (shorter SDEs) with delayed generators and tools for their analysis. Comparing to the backward case which we mentioned,  in the forward framework we refer to papers  \cite{xyh} by  Xu, Yang, Huang  and lecture notes  \cite{muh} by  Salah-Eldin with significant result in the field of processes with memory.
	
	 On the other side,  Hu and Peng introduced FBSDEs in \cite{HP}, and  proved the existence and uniqueness of the solution to forward-backward stochastic differential equations under monotonicity condition for the coefficients. There have been many papers after this result where authors generalizes form of FBSDEs, or extended the range for existence and uniqueness of the solution, and with this generalizations illustrated some new field of applications of those equations. Within  a wide range of results in the field of FBSDEs it is interesting to mention paper of Ma,  Wu,  Zhang, Zhang \cite{Ma} where authors showed well-posedness of the FBSDE in a general non-Markovian framework. This were the first steps of imposing the advantages on the other sides, but also difficulties on the other side, of non-Markovian type of equations.
	 
	 \medskip
	 
	Our paper has two goals. First, we extend work by Delong and Imkeller \cite{DI1} by studying non homogenous BSDE with delayed generator which is BSDE in the form
\begin{eqnarray}
Y(t)&=& \xi+\int^{T}_{t}f(s,Y_s,Z_s)ds-\int_{t}^{T}(g(s,Y_s)+Z(s))dW(s).\label{pl2-3}
\end{eqnarray}
This will enable us to prove the second aim which is to introduce the field of FBSDEs with delayed generators of general form. More precisely, we study the following non homogenous fully coupled forward-backward stochastic differential delayed equations (FBSDDEs in short):
\begin{eqnarray}
X(t)&=&x+\int^t_0 b(r,X_r, Y_r,Z_r)dr+\int^t_0 \sigma(r,X_r, Y_r,Z_r)dW(r),\nonumber\\\label{FBSDDE}\\\nonumber
Y(t)&=& \xi + \int_t^T f(r,X_r, Y_r,Z_r)dr - \int_t^T (g(r,X_r,Y_r)+Z(r)) d W(r)\nonumber
\end{eqnarray}
where $X_{s}=(X(s+u))_{-T\leq u\leq 0}$, $Y_{s}=(Y(s+u))_{-T\leq u\leq 0}$ and $Z_{s}=(Z(s+u))_{-T\leq u\leq 0}$ designed respectively all the past of the processes $X, Y$ and $Z$ until $s$. To our point of view, such study does not exist in literature and makes this a novelty. Even more, our method is different  from that used by Peng and Hu, and allows us to consider only the Lipschitz condition on the generators.

\medskip

The sequel of this paper is organized as follows: In Section 2 preliminary results are given.  Section 3 is dedicated existence and uniqueness theorems which are given in two subsections. Subsection 3.1 is dedicated to non homogenous BSDE with delayed generator, equation which has a general diffusion coefficient - function which is dependent of state process  is added to control process. In Theorem \ref{P1}, we present a result of existence and uniqueness for this type of equations. Subsection 3.2 is devoted to derive two results. First in Theorem \ref{T2}, we derive an existence and uniqueness for FBSDE \eqref{FBSDDE} when $g$ is identically null. Next with this previous result, we provide in Theorem \ref{T1} which establish an existence and uniqueness result for FBSDE \eqref{FBSDDE} in general case. We dedicate Section 4 to possible applications as well to some conclusion remarks. Paper finishes with a list of references.
	
\section{Preliminaries}
	For a strict positive real number $T$, let us consider $(\Omega,\mathcal{F},\P,(\mathcal{F}_t)_{0\leq t\leq T})$ a a filtered probability space, where the filtration $(\mathcal{F}_t)_{0\leq t\leq T}$ is assumed to be complete, right continuous and generated by a $(W(t))_{0\leq t\leq T}$, a one-dimensional Brownian motion.

In order to give what we mean by solution of eq. \eqref{pl2-3}, as well as eq. \eqref{FBSDDE}, let us set the following spaces;
\begin{description}
\item $\bullet$ Let $L_{-T}^2(\mathbb{R})$ denote the space of measurable functions $ z : [-T;0] \rightarrow \mathbb{R} $ satisfying
$$   \int_{-T}^0 \mid z(t) \mid^2 dt < \infty .
$$
\item $\bullet$ Let $L_{-T}^{\infty} (\mathbb{R})$ denote the space of bounded, measurable functions $y: [-T,0] \rightarrow \mathbb{R} $\\
satisfying
$$
\sup\limits_{-T\leq t\leq 0} \mid y(t) \mid^2 <+\infty.
$$

\item $\bullet$ Let $ L^2(\Omega,\mathcal{F}_T,\mathbb{P})$ be the space of $\mathcal{F}_T$-measurable random variables $\xi: \Omega \rightarrow \mathbb{R} $ endowed with the norm $$\|\xi\|_{L^2}^2:=\E(|\xi|^2).$$
\item \item $\bullet$ Let $ \mathcal{S}^2(\R)$ denote the space of all predictable process $\eta$ with values in $\R$ such that $$\E\left(\sup_{0\leq s\leq T}e^{\beta s}|\eta(s)|^2\right)<+\infty.$$
\item $\bullet$ Let $\mathcal{H}^2(\R^n)$ denote the space of all predictable process $\eta$ with values in $\R$ such that $$\E\left(\int_{0}^Te^{\beta s}|\eta(s)|^2ds\right)<+\infty.$$
\end{description}
In all this paper, we will use the following notations; $|.|$ denotes the usual norm in $\R^n$ with it associated Euclidian norm. For $u=(x,y,z)\in\R^n\times\R\times\R^n$,
$$ \|u\|^2:=|x|^2+|y|^2+|z|^2 $$
and 
$$h(t,u):=(f(t,u),b(t,u),\sigma(t,u)).$$

\medskip

We are now able to introduce the formal definition of the solution to the eqs. \eqref{pl2-3}  and  \eqref{FBSDDE}.

\smallskip

\begin{defi}
\begin{itemize}
\item [(i)] A couple of processes $(Y,Z)$ is called an adapted solution of \eqref{pl2-3}, if $(Y,Z)$ belongs to $\mathcal{S}^{2}(\R)\times \mathcal{H}^2(\R^n)$, and satisfies \eqref{pl2-3} $\P$-almost surely.	
\item [(ii)] A triple of processes $(X,Y,Z)$ is called an adapted solution of \eqref{FBSDDE}, if $(X,Y,Z)$ belongs to $\mathcal{S}^2(\R)\times\mathcal{S}^{2}(\R)\times \mathcal{H}^2(\R^n)$, and satisfies \eqref{FBSDDE} $\P$-almost surely.	 
\end{itemize}
\end{defi}

\section{Existence and uniqueness results}

This section is dedicated to the main results of the paper and results are devised in two parts. In the first subsection general BSDE with the delay with the conditions for the generators given by  \cite{DI} is introduces and existence and uniqueness is proved, while in following subsection the existence and uniqueness result  for the FBSDE with delay with appropriate conditions for the generators is introduced and problems related to it proved.

\subsection{Extension of the class of delayed backward stochastic differential equations}
This subsection is devoted to study BSDE \eqref{pl2-3} which is a more general version  of the BSDEs with time delayed generator studied by Delong and Imkeller in \cite{DI}.  The difference is that the BSDE  \eqref{pl2-3} has more general diffusion coefficient, ie it has in addition to control process $Z(\cdot)$ it has a  function which is dependent of state process $Y(\cdot)$, ie function $g(s,Y_s)$. This extended version of the equation respect to diffusion coefficient is named {\it non-homogeneous}, while the equation which has only control process in the diffusion coefficient is named {\it homogeneous} version of the equation.  This type of generalization of diffusion coefficient has been proved by Jankovi\'c, Djordjevi\'c, Jovanovi\'c for backward doubly stochastic differential equations (shorter BDSDEs) under different conditions for the coefficients of the equation in \cite{J1} and \cite{J3}. Also, in  \cite{J2}   Djordjevi\'c and  Jankovi\'c gave the result of this type of extension for the backward stochastic Volterra integral equations (shorter BSVIEs).

\medskip

Our main objective is to derive an existence and uniqueness result of BSDE \eqref{pl2-3}, that extends Theorem 3.1 established in \cite{DI}, under the following assumptions;
 \begin{description}
\item {\bf(H1)} $\xi \in L^2(\Omega,\mathcal{F}_T,\mathbb{P})$.
\item {\bf(H2)} $f: \Omega \times[0, T] \times L_{-T}^{\infty}(\mathbb{R}^n) \times L_{-T}^{2}(\mathbb{R}^{n\times d}) \rightarrow \mathbb{R}^n$ is a product measurable and $ {\bf F} $-adapted function such that for some probability measure  $\alpha$ on $([-T, 0],\mathcal{B}([-T, 0]))$ and any $m_t=\left(y_{t}, z_{t}\right),m'_t=\left(y'_{t}, z'_{t}\right) \in L_{-T}^{\infty}(\mathbb{R}^n) \times L_{-T}^{2}(\mathbb{R}^{n\times d})$, there exist a positive constant $K$ such that following holds:
 \begin{itemize}
 	\item [(i)]$\displaystyle |f\left(t, m_t\right)-f(t, m'_{t})|^{2}\leq K\int_{-T}^{0}\|m(t+u)-m'(t+u)\|^{2}\alpha(du)$ a.s.
 	\item [(ii)] $\displaystyle \mathbb{E}\left[\int_{0}^{T}|f(t, 0,0)|^{2} d t\right]<+\infty$
 	\item [(iii)]\ $f(t,\cdot,\cdot)=0$ for $t<0$
 \end{itemize}
\item {\bf(H3)} $g:\Omega \times[0, T] \times L_{-T}^{\infty}(\mathbb{R}^n) \rightarrow \mathbb{R}^{n\times d}$ is a product measurable and $ {\bf F} $-adapted function such that for some probability measure $\alpha$ on $([-T, 0],\mathcal{B}([-T, 0]))$ and any  $y_{t},y'_t\in L_{-T}^{\infty}(\mathbb{R}) \times L_{-T}^{2}(\mathbb{R})$, there exist a positive constant $K$ such that following holds:
\begin{itemize}
	\item [(i)] $\displaystyle \left|g\left(t, y_{t}\right)-g\left(t, y'_{t}\right)\right|^{2}\leq  K\int_{-T}^{0}|y(t+u)-y'(t+u)|^{2} \alpha(du)$ a.s.,
	\item [(ii)] $\displaystyle \mathbb{E}\left[\int_{0}^{T}|g(t, 0)|^{2} d t\right]<+\infty$,
	\item [(iii)] $g(t, \cdot)=0$ for $t<0$.
\end{itemize}
\end{description}
We conclude by explaining some of assumptions. 
\begin{rem}\label{R1}
\begin{itemize}
\item [(a)] Assumption $(\bf H2)$-$(ii)$ and $(\bf H3)$-$(ii)$ allow us to take $(X(t),Y(t),Z(t))=(X(0),Y(0),0)$ for $t<0$, as a solution of \eqref{pl2-3}.
\item [(b)] The quantity $f(t,0,0)$ in $ (\bf H2)$-$(iii) $ and $g(t,0)$ in $ (\bf H3)$-$(iii) $ should be understood as a value of the generator $f$ (resp. $g$) at $m_t=(0,0)$ (resp. $y_t=0$).
\end{itemize}
\end{rem}

\begin{theo}\label{P1}
Assume ${\bf(H1)-(H3)}$ hold. Let horizon time $T>0$ and Lipschitz constant $K$ satisfy
\begin{eqnarray*}
8KT \max (1,T)<1. 
\end{eqnarray*}
Then delayed BSDE \eqref{pl2-3} has a unique solution.	
\end{theo}

\begin{proof}
Our method follows the same idea used by Jankovi\'c, \DJ or\dj evi\'c and Jovanovi\'c \cite{J1,J2,J3}. In this fact, let us consider this following BSDE
 \begin{eqnarray}
Y(t)&=& \xi+\int^{T}_{t}\bar{f}(s,Y_s,Z_s)ds-\int_{t}^{T}Z(s)dW(s),\label{AAA}
\end{eqnarray}
where $\bar{f}(t,y_t,z_t)=f(t,y_t,z_t-g(t,y_t))$. 
In view of Assumption $(\bf H2)$, the function $\bar{f}$ satisfies the following assumption
\begin{description}
\item {\bf(H4)}$\bar{f}:\Omega \times[0,T]\times L_{-T}^{\infty}(\mathbb{R})\times L_{-T}^{2}(\mathbb{R}) \rightarrow \mathbb{R}$ is a product measurable and $ {\bf F} $-adapted function such that for some probability measure  $\alpha $ on $([-T, 0],\mathcal{B}([-T, 0]))$ and any $\left(y_{t}, z_{t}\right),\left(y'_{t}, z'_{t}\right) \in L_{-T}^{\infty}(\mathbb{R}^n) \times L_{-T}^{2}(\mathbb{R}^{n\times d})$, there exist a positive constant $K$ such that following holds:
\begin{itemize}
	\item [(i)]$\displaystyle |\bar{f}\left(t, m_{t}\right)-\bar{f}(t, m'_{t})|^{2}\leq K\int_{-T}^{0}\|m(t+u)-m'(t+u)\|^{2}\alpha(du)$ a.s.,
	\item [(ii)] $\displaystyle \mathbb{E}\left[\int_{0}^{T}|\bar{f}(t, 0,0)|^{2} d t\right]<+\infty$
	\item [(iii)] $\bar{f}(t,\cdot,\cdot)=0$ for $t<0$.
\end{itemize}
\end{description}
Therefore, according to Theorem 2.1 in \cite{DI}, there exist a unique process $(\overline{Y},\overline{Z})$, solution of BSDE \eqref{AAA}.
It remain to prove that BSDE \eqref{pl2-3} has a unique solution $(Y,Z)$ such that $Y:=\overline{Y},$ and $Z:=\overline{Z}-g(.,\overline{Y})$. 
In this fact, let consider $(Y',Z')$ an other solution of BSDE \eqref{pl2-3}. Next setting $\delta Y=Y'-\overline{Y}$, it follows from \eqref{pl2-3} and \eqref{AAA} that  
\begin{eqnarray}\label{eqq3}
\delta Y (t)&=& \int^{T}_{t}\left[f(s,Y'_s,Z'_s)-f(s,\overline{Y}_s,\overline{Z}_s-g(s,\overline{Y}_s))\right]ds \nonumber \\
&&-\int_{t}^{T}\left[Z'(s)+g(s,Y'_s)-\overline{Z}(s))\right]dW(s).
\end{eqnarray}
Hence, taking the conditional expectation which respect to $\mathcal{F}_t$ in both side of \eqref{eqq3}, we obtain
\begin{eqnarray*}
|\delta Y(t)|\leq  \E\left[\left|\int_0^T \left[f(s,Y'_s,Z'_s)-f(s,\overline{Y}_s,\overline{Z}_s-g(s,\overline{Y}_s))\right] ds\right| \mid \mathcal{F}_t \right].
\end{eqnarray*}
Moreover by Doob's inequality together with Jensen's inequality, we have
\begin{eqnarray}\label{z}
\E\left[\sup\limits_{0 \leqslant t \leqslant T}|\delta Y(t)|^2\right]
&\leqslant& 4\sup\limits_{0 \leqslant t \leqslant T} \E\left[\E\left(\left|\int_0^T \left[f(s,Y'_s,Z'_s)-f(s,\overline{Y}_s,\overline{Z}_s-g(s,\overline{Y}_s))\right] ds\right|^2 \mid\mathcal{F}_t\right)\right].\nonumber\\
&\leq 4 &\E\left(\left|\int_0^T \left[f(s,Y'_s,Z'_s)-f(s,\overline{Y}_s,\overline{Z}_s-g(s,\overline{Y}_s))\right] ds\right|^2\right).
\end{eqnarray}
Recall again \eqref{eqq3}, we have
\begin{eqnarray*}
&&\left( \int_{t}^{T}\left[Z'(s)+g(s,Y'_s)-\overline{Z}(s))\right]dW(s) \right)^2 \\
& \leq & 2\left|\int_0^T \left[f(s,Y'_s,Z'_s)-f(s,\overline{Y}_s,\overline{Z}_s-g(s,\overline{Y}_s))\right] ds   \right|^2 +2| \delta Y(0)|^2.
\end{eqnarray*}
The expectation in both side of previous inequality together with previous estimate provide
\begin{eqnarray}
&& \E\left(\int_0^T|Z'(s)+g(s,Y'_s)-\overline{Z}(s))|^2ds\right) \nonumber \\
&\leq & 2\E\left( \left|\int_0^T \left[f(s,Y'_s,Z'_s)-f(s,\overline{Y}_s,\overline{Z}_s-g(s,\overline{Y}_s))\right] ds   \right|^2 + |\delta Y(0)|^2\right).\nonumber\\
&\leq & 4\E\left(\left|\int_0^T \left[f(s,Y'_s,Z'_s)-f(s,\overline{Y}_s,\overline{Z}_s-g(s,\overline{Y}_s))\right] ds   \right|^2 \right).\label{j}
\end{eqnarray}
Hence, in virtue of \eqref{z} and \eqref{j}, we get
\begin{eqnarray}\label{F}
&&\E\left[\sup\limits_{0 \leqslant t \leqslant T}|\delta Y(t)|^2\right]+\E\left(\int_0^T|Z'(s)+g(s,Y'_s)-\overline{Z}(s))|^2ds\right)\nonumber\\
&\leq& 8\E\left(\left|\int_0^T \left[f(s,Y'_s,Z'_s)-f(s,\overline{Y}_s,\overline{Z}_s-g(s,\overline{Y}_s))\right] ds   \right|^2 \right).
\end{eqnarray}
But, it follows from the Lipschitz condition $({\bf A2})$ on generator $f$ and Cauchy-Schwarz's inequality that,
\begin{eqnarray}\label{k}
&& \E\left(\left|\int_0^T \left[f(s,Y'_s,Z'_s)-f(s,\overline{Y}_s,\overline{Z}_s-g(s,\overline{Y}_s))\right] ds\right|^2\right) \nonumber\\
&\leqslant & T\E\left(\int_0^T \left|f(s,Y'_s,Z'_s)-f(s,\overline{Y}_s,\overline{Z}_s-g(s,\overline{Y}_s))\right|^2ds     \right) \nonumber \\
&\leq & KT\E\left(\int^{T}_{0} \int_{-T}^{0}|\delta Y(s+u)|^{2} \alpha(d u) ds\right) \nonumber\\ 
&&+ KT\E\left(\int^{T}_{0} \int_{-T}^{0}|Z'(t+u)+g(s+u,Y'(s+u))-\overline{Z}(s+u)|^{2} \alpha(du)ds\right).
\end{eqnarray}
Since $Y(s)=Y(0)$ and $Z(s)=0=g(s,.)$ for all $s<0$, we obtain respectively by Fubini's theorem and the change of variable
\begin{eqnarray}\label{eqq9}
 \E \left( \int^{T}_{0} \int_{-T}^{0}|\delta Y(s+u)|^{2} \alpha(d u) ds\right) 
& \leq &  \E  \left(\int^{0}_{-T} \int_{u}^{T+u}|\delta Y(v)|^{2} dv\alpha ( du) \right) \nonumber \\
& \leq & T \E \left(  \sup\limits_{0 \leqslant v \leqslant T}\mid \delta Y(v)\mid^2 \right)
\end{eqnarray}
and
\begin{eqnarray}\label{eqq10}
&& \E\left(\int^{T}_{0} \int_{-T}^{0}|Z'(t+u)+g(s+u,Y'(s+u))-\overline{Z}(s+u)|^{2} \alpha(du)ds\right) \nonumber\\
& \leq &  \E  \left( \int^{0}_{-T} \int_{u}^{T+u} |Z'(v)+g(v,Y'(v))-\overline{Z}(v)|^{2}dv \alpha (du)  \right) \nonumber\\
& \leq & \E  \left(  \int_{0}^{T} |Z'(v)+g(v,Y'(v))-\overline{Z}(v)|^{2}dv  \right).
\end{eqnarray}
Putting \eqref{eqq9}, \eqref{eqq10} and \eqref{k} in \eqref{F}, we have
\begin{eqnarray*}
&& \E\left(  \sup\limits_{0 \leqslant t \leqslant T}\mid \delta Y(t)\mid^2 + \int_0^T|Z'(s)+g(s,Y'_s)-\overline{Z}(s)|^2 ds \right)\\
& \leq & 8KT \max (1,T)  \E\left(  \sup\limits_{0 \leqslant t \leqslant T}\mid \delta Y(t)\mid^2 +  \int_{0}^{T} |Z'(s)+g(s,Y'(s))-\overline{Z}(s)|^{2}ds\right).
\end{eqnarray*}
Since $ 8KT \max (1,T)< 1 $,
\begin{eqnarray*}
\E\left(  \sup\limits_{0 \leqslant t \leqslant T}\mid \delta Y(t)\mid^2 +  \int_{0}^{T} |Z'(s)+g(s,Y'(s))-\overline{Z}(s)|^{2}ds\right)= 0.
\end{eqnarray*}
Thus
$$Y'(t)=\overline{Y}(t) \quad  a.s.,  \qquad\qquad  Z'(t)=\overline{Z}(t)-g(t,\overline{Y}_t)\quad a.s. \qquad t\in[0,T].$$
This completes the proof.
\end{proof}
\subsection{Delayed forward backward stochastic differential equations}
This subsection is devoted to give the second main result of our paper, which is the existence and uniqueness of the solution to FBSDDEs of the form \eqref{FBSDDE}. Our method differ to one applied by Hu and Peng in \cite{HP}. Here we have applied It\^o's formula to $|X|^2$ and $|Y|^2$ rather than $\langle X,Y\rangle $ which permit us to relax some of their assumptions. Indeed, we work under the following classical assumptions.
\begin{description}

\item {(\bf A1)} $\xi \in L^2(\Omega,\mathcal{F}_T,\mathbb{P})$, 
\item {(\bf A2)} $ \Phi: \Omega \times [0,T]\times L_{-T}^2 (\mathbb{R})\times  L_{-T}^{\infty} (\mathbb{R} ) \times  L_{-T}^2 (\mathbb{R}^n ) \rightarrow \mathbb{R} $ is a product measurable and $ {\bf F} $-adapted function such that for some  probability measure $\alpha$ defined on $([-T,0],\mathcal{B}([-T,0]))$ and  any $ u_t=(x_t,y_t,z_{t}),u'_t=(x'_t,y'_t,z'_{t}) \in  L_{-T}^2 (\mathbb{R})\times L_{-T}^{\infty} (\mathbb{R} ) \times  L_{-T}^2 (\mathbb{R}^n)$,  there exist two positive constant $K$ such that following holds:
\begin{itemize}
	\item [(i)]
$\displaystyle
\mid \Phi(t,u_t) - \Phi(t,u'_t )\mid^2 
\leq  K \int_{-T}^0  \| u(t+v) - u'(t+v) \|^2  \alpha (dv)$ a.s.\ ,
\item [(ii)]For $t<0,\; \; \Phi(t,u_t)= 0 $,
\item [(iii)]$\displaystyle \mathbb{E} \left[ \int_{0}^T \vert \Phi(t,0) \vert^2 dt\right]  <  +\infty$, 
\end{itemize}
where $\Phi=b,\sigma,f$.
\item \item {(\bf A3)} $ g: \Omega \times [0,T]\times L_{-T}^2 (\mathbb{R})\times  L_{-T}^{\infty} (\mathbb{R} ) \rightarrow \mathbb{R} $ is a product measurable and $ {\bf F} $-adapted function such that for some  probability measure $\alpha$ defined on $([-T,0],\mathcal{B}([-T,0]))$ and  any $ (x_t,y_t),(x'_t,y'_t) \in  L_{-T}^2 (\mathbb{R})\times L_{-T}^{\infty} (\mathbb{R} )$,  there exist positive constant $K$ such that following holds:
\begin{itemize}
	\item [(i)]
$\displaystyle
\mid g(t,x_t,y_t) - g(t,x'_t,y'_t)\mid^2 
\leq  K \int_{-T}^0  (| x(t+u) - x'(t+u)|^2+|y(t+u)-y'(t+u)|^2) \alpha (du)$ a.s.\ ,
\item [(ii)]For $t<0,\; \; g(t,x_t,y_t)= 0 $,
\item [(iii)]$\displaystyle \mathbb{E} \left[ \int_{0}^T \vert g(t,0) \vert^2 dt\right]<  +\infty$, 
\end{itemize}
\end{description}

Next remark is in accordance with Remark \ref{R1}.
\begin{rem}\label{R1}
\begin{itemize}
\item [(a)] Assumption $(\bf A2)$-$(ii)$ allows us to take $(X(t),Y(t),Z(t))=(X(0),Y(0),0)$ for $t<0$, as a solution of \eqref{FBSDDE}.
\item [(b)] The quantity $\phi(t,0)$ in $ (\bf A2)$-$(iii) $ and $g(t,0)$ in $ (\bf A3)$-$(iii) $ should be understood respectively as a value of the generator $\phi$ at $u_t=(0,0,0)$ and $g$ at $(x_t,y_t)=(0,0)$.
\end{itemize}
\end{rem} 
\begin{theo}\label{T1}
Let assumptions $({\bf A1})$-$({\bf A3})$ hold. Let horizon time $T>0$ and Lipschitz constant $K$ satisfy
\begin{eqnarray*}
 24 K e\max(1,T^2)<1. 
\end{eqnarray*}
 Then there exists a unique adapted solution $(X, Y, Z)$ for delayed FBSDEs \eqref{FBSDDE}.
\end{theo}
To prove this results, let us first establish the existence and uniqueness result for particular FBSDE \eqref{FBSDDE} ($g$ is identically null) in the following form
\begin{eqnarray}
X(t)&=&x+\int^t_0 b(r,X_r, Y_r,Z_r)dr+\int^t_0 \sigma(r,X_r, Y_r,Z_r)dW(r),\nonumber\\\label{FBSDDEbis}\\\nonumber
Y(t)&=& \xi + \int_t^T f(r,X_r, Y_r,Z_r)dr - \int_t^T Z(r) d W(r)
\end{eqnarray}
\begin{theo}\label{T2}
Assume $({\bf A1})$-$({\bf A2})$ hold. Let horizon time $T>0$ and Lipschitz constant $K$ satisfy
\begin{eqnarray*}
 24 K e\max(1,T^2)<1. 
\end{eqnarray*}
 Then there exists a unique adapted solution $(X, Y, Z)$ for delayed FBSDEs \eqref{FBSDDEbis}.
\end{theo}

\begin{proof} This proof is subdivided in two parts. 

\bigskip

{\bf Step 1:} {\it Existence.}\\

Let set $U^0=(X^0,Y^0,Z^0)=(0,0,0)$ and consider $U^{n}=(X^{n}, Y^{n},Z^{n})$ defined recursively as follows
\begin{eqnarray}\label{K1}
 X^n(t) &=&  x+\int^t_0 b (s,U^{n-1}_s) ds+ \int^t_0\sigma (s,U^{n-1}_s) dW(s),\nonumber \\\\
Y^n(t) &=&  \xi + \int_t^T f (s,U^{n-1}_s)ds - \int_t^T Z^n(s) d W(s).\nonumber
\end{eqnarray}

Setting $$\bar{\phi }^{n+1} = \phi^{n+1} -\phi^n,$$ for $ \phi $ equal to $ X ,Y,Z $ and $$\bar{\Phi} (s)=\Phi (s,X^{n}_s, Y^{n}_s,Z^{n}_s) -\Phi (s,X^{n-1}_s, Y^{n-1}_s,Z^{n-1}_s),$$ for $ \Phi $ equal to $ b ,\sigma,f $, we have
\begin{eqnarray}\label{K2}
&&\bar{X}^{n+1} (t) = \int^t_0  \bar{b} (s)  ds + \int^t_0  \bar{\sigma } (s)  dW(s)\nonumber \\\\
&&\bar{Y}^{n+1} (t) = \int_t^T  \bar{f} (s) ds - \int_t^T \bar{Z}^{n+1}(s)d W(s).\nonumber
\end{eqnarray}
Applying It\^o's formula to $ e^{\frac{\beta}{2}t} \bar{X}^{n+1}(t)$ and  then taking conditional expectation which respect $\mathcal{F}_t$ , we have

\begin{eqnarray*}
e^{\frac{\beta}{2}t} \mid  \bar{X}^{n+1} (t) \mid \leq  \E \left(\frac{\beta}{2} \int_0^t e^{\frac{\beta}{2}s}\mid  \bar{X}^{n+1} (s) \mid ds + \int_0^t e^{\frac{\beta}{2}s}\mid \bar{b} (s)\mid ds +\int_0^te^{\frac{\beta}{2}s}\bar{\sigma}(s)dW(s)\mid \mathcal{F}_t \right) 
\end{eqnarray*}
By Doob's martingale inequality and Cauchy-Schwarz's inequality we
\begin{eqnarray}\label{Di1}
&&\E \left( \sup\limits_{0 \leq t \leqslant T} e^{\beta t} \mid \bar{X}^{n+1} (t) \mid^2 \right)\nonumber\\
&\leq &
\frac{\beta^2}{2} T \int_0^T \E \left(   e^{\beta s} \mid \bar{X}^{n+1} (s) \mid^2 \right)ds + 4\E \left(T \int_0^T e^{\beta s}\mid \bar{b} (s)\mid^2 ds+\int_0^te^{\beta s}|\bar{\sigma}(s)|^2 ds\right)\nonumber\\
&\leq & \frac{\beta^2}{2} T^2\E \left( \sup_{0\leq t\leq T}  e^{\beta t} \mid \bar{X}^{n+1} (t) \mid^2\right) + 4\E \left(T\int_0^T e^{\beta s}\mid \bar{b} (s)\mid^2 ds +\int_0^te^{\beta s}|\bar{\sigma}(s)|^2 ds\right).\nonumber\\
\end{eqnarray}
Hence we have
\begin{eqnarray}\label{Di3}
\E \left( \sup\limits_{0 \leq t \leqslant T} e^{\beta t} \mid \bar{X}^{n+1} (t) \mid^2 \right)
&\leq &4\left(1-\frac{\beta^2}{2}T^2\right)^{-1}\max(1,T)\E\left(\int_0^T e^{\beta s}\mid \bar{b} (s)\mid^2 ds+\int_0^Te^{\beta s}|\bar{\sigma}(s)|^2 ds\right).\nonumber\\
\end{eqnarray}
Let us now treat the backward SDE. For this by applying again It\^o's formula to $ e^{\beta t} \mid \bar{Y}^{n+1} (t) \mid^2$, we have 
\begin{eqnarray*}
&&\E \left(  e^{\beta t} \mid \bar{Y}^{n+1} (t) \mid^2 + \beta  \int_t^T e^{\beta s} \mid \bar{Y}^{n+1} (s) \mid^2 ds + \int_t^T e^{\beta s} \mid \bar{Z}^{n+1} (s) \mid^2 ds \right)\nonumber\\
&&= 2 \E \int_t^T e^{\beta s} \bar{Y}^{n+1} (s)\bar{f} (s) ds.
\end{eqnarray*}

Since it follows from Young's inequality there a constant $\beta$ such that
\begin{eqnarray*}
2\bar{Y}^{n+1} (s)\bar{f} (s)\leq \beta |\bar{Y}^{n+1} (s)|^2+\frac{1}{\beta}|\bar{f} (s)|^2,	
\end{eqnarray*}
we obtain
\begin{eqnarray}\label{Di2}
\E \left(  e^{\beta t} \mid \bar{Y}^{n+1} (t) \mid^2  + \int_t^T e^{\beta s} \mid \bar{Z}^{n+1} (s) \mid^2 ds \right)&=& \frac{1}{\beta}\E \int_t^T|\bar{f} (s)|^2 ds.
\end{eqnarray} 

On the other hand, applying again It\^o's formula to $e^{\frac{\beta}{2}t} \bar{Y}^{n+1}(t)$ and conditional expectation which respect $\mathcal{F}_t$ respectively, we have
\begin{eqnarray*}
e^{\frac{\beta}{2}t}\bar{Y}^{n+1} (t) = \E \left(\int_t^T  e^{\frac{\beta}{2}s}\bar{f} (s) ds \mid \mathcal{F}_t \right).
\end{eqnarray*}

Using again Doob's martingale inequality and Cauchy-Schwarz's inequality  we obtain
\begin{eqnarray}\label{DI4}
\E \left( \sup\limits_{0 \leq t \leqslant T} e^{\beta t} \mid \bar{Y}^{n+1} (t) \mid^2 \right) \leq T \E \left( \int_0^T e^{\beta s}\mid \bar{f} (s)\mid^2 ds  \right)
\end{eqnarray}
Next, it follows from \eqref{DI4},\eqref{Di2} and \eqref{Di3} that
\begin{eqnarray*}\label{DI5}
&&\E \left[ \sup\limits_{0 \leq t \leqslant T}e^{\beta t}\mid \bar{X}^{n+1} (t) \mid^2 + \sup\limits_{0 \leq t \leqslant T} e^{\beta t}\mid \bar{Y}^{n+1} (t)\mid^2 + \int_0^T e^{\beta s}\mid \bar{Z}^{n+1} (s)\mid^2 ds \right]\nonumber\\
&\leq & \mathbb{E} \left[4\left(1-\frac{\beta^2}{2}T^2\right)^{-1}\max(1,T)\left(\int^T_0 e^{\beta s}|\bar{b} (s)|^2 ds+\int^T_0 e^{\beta s}|\bar{\sigma} (s)|^2 ds\right)\right.\\
&&\left.+\left(T+\frac{1}{\beta}\right)\int_0^Te^{\beta s}| \bar{f} (s)|^2 ds \right].\nonumber\\   
\end{eqnarray*}
For $\beta=\frac{1}{T}$, we have
\begin{eqnarray}
\hskip -1.6cm&&\E \left[ \sup\limits_{0 \leq t \leqslant T}\mid e^{\beta t}\bar{X}^{n+1} (t) \mid^2 + \sup\limits_{0 \leq t \leqslant T} e^{\beta t}\mid \bar{Y}^{n+1} (t)\mid^2 + \int_0^Te^{\beta s}\mid \bar{Z}^{n+1} (s)\mid^2 ds \right]\nonumber\\
\hskip -1.2cm&\leq & 8\max(1,T)\mathbb{E} \left[\int^T_0 e^{\beta s}|\bar{b} (s)|^2 ds+\int^T_0 e^{\beta s}|\bar{\sigma} (s)|^2 ds+\int_0^T e^{\beta s}| \bar{f} (s)|^2 ds \right].\label{DI5}  
\end{eqnarray}

It remains to estimate the second side of the inequality \eqref{DI5}. For this instance, it follows from the Lipschitz condition $({\bf A2})$ on $b$, $\sigma$ and $f$ that,
\begin{eqnarray*}
&&\mathbb{E} \left[  \int^T_0 e^{\beta s}|\bar{\Phi}(s)|^2 ds  \right]\\
&\leq & K \E \left[ \int_0^T \int_{-T}^0 e^{\beta s}\left( \mid \bar{X}^n (s+u) \mid^2 + \mid \bar{Y}^n (s+u) \mid^2 + \mid \bar{Z}^n (s+u) \mid^2 \right)\alpha (du) ds \right].
\end{eqnarray*}
Since for each $n\geq 1$, $X^n(s) = 0,\; Y^n(s)=Y^n(0)$ and $ Z^n(s)=0$ for all $ s<0 $, we obtain respectively by Fubini's theorem and the change of variable
\begin{eqnarray*}
&&\mathbb{E} \left[\int^T_0 e^{\beta t}|\bar{\Phi} (s)|^2 ds\right] \\
&\leq &  K \E \left[  \int_{-T}^0 e^{-\beta u}\left(\int_u^{T+u}e^{\beta v}\left( \mid \bar{X}^n (v) \mid^2 + \mid \bar{Y}^n (v) \mid^2 + \mid \bar{Z}^n (v) \mid^2 \right) dv\right)\alpha (du)  \right]\\
&\leq & K e\max(1,T)\E \left[\sup\limits_{0 \leq t \leqslant T}e^{\beta t}\mid \bar{X}^n (t) \mid^2 + \sup\limits_{0 \leq t \leqslant T}e^{\beta t}\mid \bar{Y}^n (t) \mid^2 + \int_0^T e^{\beta s}\mid \bar{Z}^n (s) \mid^2 ds \right].
\end{eqnarray*}
Consequently, it follow from the previous inequality together with \eqref{DI5} that
\begin{eqnarray*}
&&\E \left[ \sup\limits_{0 \leq t \leqslant T}e^{\beta t}\mid \bar{X}^{n+1} (t) \mid^2 + \sup\limits_{0 \leq t \leqslant T} e^{\beta t}\mid \bar{Y}^{n+1} (t)\mid^2 + \int_0^T e^{\beta s}\mid \bar{Z}^{n+1} (s)\mid^2 ds \right]  \\
&\leq & 24 K e\max(1,T^2)\E \left[\sup\limits_{0 \leq t \leqslant T}e^{\beta t}\mid \bar{X}^n (t) \mid^2 +\sup\limits_{0 \leq t \leqslant T}e^{\beta t}\mid\bar{Y}^n (t) \mid^2+\int_0^T e^{\beta s}\mid \bar{Z}^n (s) \mid^2 ds \right].
\end{eqnarray*}
Finally, by iterative argument, it not difficult to derive
\begin{eqnarray*}
&&\E \left[ \sup\limits_{0 \leq t \leqslant T}\mid e^{\beta t}\bar{X}^{n+1} (t) \mid^2 + \sup\limits_{0 \leq t \leqslant T} e^{\beta t}\mid \bar{Y}^{n+1} (t)\mid^2 + \int_0^T e^{\beta s}\mid \bar{Z}^{n+1} (s)\mid^2 ds \right]  \\
&&\leq \left[ 24 K e\max(1,T^2)\right]^n  \E \left[   \sup\limits_{0 \leq t \leqslant T}\mid e^{\beta t}\bar{X}^1 (t) \mid^2 +   \sup\limits_{0 \leq t \leqslant T}e^{\beta t}\mid \bar{Y}^1 (t) \mid^2 + \int_0^T e^{\beta s}\mid \bar{Z}^1 (s) \mid^2 ds \right].
\end{eqnarray*}
Since $  24 K e\max(1,T^2)< 1, \; (X^n,Y^n,Z^n)_{n\geq 1}$ is a Cauchy sequence in the Banach space $\mathcal{S}^2 (\R) \times \mathcal{S}^2 (\R)\times \mathcal{H}^2 (\R)$. It is then easy to conclude that $\displaystyle(X,Y,Z)=\lim_{n\rightarrow +\infty}(X^n,Y^n,Z^n)$ solves \eqref{FBSDDE}.

\bigskip

{\bf Step 2:} {\it Uniqueness.}

\bigskip

Let define $U=(X,Y,Z) $ and $U'=(X',Y',Z') $ be two solutions of equation \eqref{FBSDDEbis}. Let us set $\delta \phi=\phi-\phi'$, for $\phi=X,Y,Z$. Then $(\delta X, \delta Y,\delta Z)$ satisfies
\begin{eqnarray*}
\delta X (t) &=& \int^t_0  \bar{b} (s)ds+\int^t_0  \bar{\sigma} (s)dW(s)\\
\delta Y (t) &=& \int_t^T  \bar{f} (s)ds+-\int_t^T  \delta Z (s)dW(s),
\end{eqnarray*}
where $\bar{\Phi}(s)=\Phi(s,X_s,Y_s,Z_s)- \Phi(s,X'_s,Y'_s,Z'_s)$ for $\Phi=b, \sigma$ and $f$. With the similar steps used in existence part, we obtain
\begin{eqnarray*}
&&\E \left[ \sup\limits_{0 \leq t \leqslant T}e^{\beta t}\mid \delta{X} (t) \mid^2 + \sup\limits_{0 \leq t \leqslant T}e^{\beta t}\mid \delta{Y}(t)\mid^2 + \int_0^T e^{\beta s}\mid \delta{Z}(s)\mid^2 ds \right]\nonumber\\
&\leq &  24 K e\max(1,T^2)\E \left[\sup\limits_{0 \leq t \leqslant T}e^{\beta t}\mid \delta{X} (t) \mid^2 +\sup\limits_{0 \leq t \leqslant T}e^{\beta t}\mid\delta{Y}(t) \mid^2+\int_0^T e^{\beta s}\mid \delta{Z}(s) \mid^2 ds \right].  
\end{eqnarray*}
In the fact that $24 Ke\max(1,T^2)<1$ we have
\begin{eqnarray*}
\E \left[ \sup\limits_{0 \leq t \leqslant T}e^{\beta t}\mid \delta{X} (t) \mid^2 + \sup\limits_{0 \leq t \leqslant T} e^{\beta t}\mid \delta{Y}(t)\mid^2 + \int_0^Te^{\beta s}\mid \delta{Z}(s)\mid^2 ds \right]
&\leq &  0,  
\end{eqnarray*}
which implies that 
$X=X',\; Y=Y'$ and $Z=Z'$.
\end{proof}
{\bf Proof of Theorem \ref{T1}}\\
Since the proof 
\begin{eqnarray}
X(t)&=&x+\int^t_0 \widetilde{b}(r,X_r, Y_r,Z_r)dr+\int^t_0 \widetilde{\sigma}(r,X_r, Y_r,Z_r)dW(r),\nonumber\\\label{FBSDDE2}\\\nonumber
Y(t)&=& \xi + \int_t^T \widetilde{f}(r,X_r, Y_r,Z_r)dr - \int_t^T Z(r) d W(r),
\end{eqnarray}
where $\widetilde{\Phi}(t,x_t,y_t,z_t)=\Phi(t,x_t,y_t,z_t-g(t,x_t,y_t))$, with $\Phi=b,\sigma, f$.

According to Assumption $(\bf A2)$ and $(\bf A3)$, the function $\bar{\Phi}$ satisfies the following assumption
\begin{description}
\item {\bf(A4)}  $\widetilde{\Phi}: \Omega \times[0, T] \times L_{-T}^{\infty}(\mathbb{R}) \times L_{-T}^{2}(\mathbb{R}) \rightarrow \mathbb{R}$ is a product measurable and $ {\bf F} $-adapted function such that for some probability measure  $\alpha$ on $([-T, 0],\mathcal{B}([-T, 0]))$ and for any $u(t)=\left(x_t,y_{t}, z_{t}\right),u'(t)=\left(x'_t,y'_{t}, z'_{t}\right) \in L_{-T}^{\infty}(\mathbb{R}^n)\times L_{-T}^{\infty}(\mathbb{R}^n) \times L_{-T}^{2}(\mathbb{R}^{n\times d})$ there exists a positive constant $K$ such that following holds:
\begin{itemize}
	\item [(i)] $\displaystyle\widetilde{\Phi}\left(t, y_{t}, z_{t}\right)-\widetilde{\Phi}(t, y'_{t}, z'_{t})|^{2}\leq K\int_{-T}^{0}\|u(t+v)-u'(t+v)\|^{2}\alpha(dv)$ a.s.,
	\item [(ii)] $\displaystyle \mathbb{E}\left[\int_{0}^{T}|\widetilde{\Phi}(t, 0,0,0)|^{2}dt\right]<+\infty$
	\item [(iii)] $\widetilde{\Phi}(t,\cdot,\cdot,\cdot)=0$ for $t<0$.
\end{itemize}
\end{description}
Next, in view of Theorem \ref{T2}, there exist a unique triple of processes $(\overline{X},\overline{Y},\overline{Z})$ solution of FBSDDE \eqref{FBSDDE2}. Let now set $X=\overline{X},\;Y=\overline{Y},\;Z=\overline{Z}-g(.,\overline{X},\overline{Y})$. Then it not difficult to prove that 
$(X,Y,Z)$ is the solution of FBSDDE \eqref{FBSDDE}. It remain to show that such solution is unique.

In this fact, let us consider $(X',Y',Z')$ an other solution of FBSDDE \eqref{FBSDDE}. Setting $\Delta X=\overline{X}-X',\; \Delta Y =\overline{Y}-Y'$,    
we have
\begin{eqnarray*}
\Delta X (t)&=& \int^{t}_{0}\left[b(s,Y'_s,Z'_s)-b(s,\overline{Y}_s,\overline{Z}_s-g(s,\overline{X}_s,\overline{Y}_s))\right]ds \nonumber \\
&&\int_0^t\left[\sigma(s,Y'_s,Z'_s)-\sigma(s,\overline{Y}_s,\overline{Z}_s-g(s,\overline{Y}_s))\right]dW(s)\\
\Delta Y (t)&=& \int^{T}_{t}\left[f(s,Y'_s,Z'_s)-f(s,\overline{Y}_s,\overline{Z}_s-g(s,\overline{X}_s,\overline{Y}_s))\right]ds \nonumber \\
&&-\int_{t}^{T}\left[Z'(s)+g(s,Y'_s)-\overline{Z}(s))\right]dW(s).
\end{eqnarray*}
It follows from the proof of uniqueness of Theorem \ref{T1} that
\begin{eqnarray*}
&&\E \left[ \sup\limits_{0 \leq t \leqslant T}e^{\beta t}\mid \Delta{X} (t) \mid^2 + \sup\limits_{0 \leq t \leqslant T}e^{\beta t}\mid \Delta{Y}(t)\mid^2 + \int_0^T e^{\beta s}\mid Z'(s)+g(s,X'_s,Y'_s)-\overline{Z}(s)) \mid^2 ds \right]\nonumber\\
&\leq &  24 K e[\max(1,T)]^2\E \left[\sup\limits_{0 \leq t \leqslant T}e^{\beta t}\mid \Delta{X} (t) \mid^2 +\sup\limits_{0 \leq t \leqslant T}e^{\beta t}\mid\Delta{Y}(t) \mid^2\right.\\
&&\left.+\int_0^T e^{\beta s}\mid Z'(s)+g(s,X'_s,Y'_s)-\overline{Z}(s)) \mid^2 ds \right].  
\end{eqnarray*}
Suppose again that $ 24 K e\max(1,T^2)<1$, we have
\begin{eqnarray*}
&&\E \left[ \sup\limits_{0 \leq t \leqslant T}e^{\beta t}\mid \Delta{X} (t) \mid^2 + \sup\limits_{0 \leq t \leqslant T}e^{\beta t}\mid \Delta{Y}(t)\mid^2 + \int_0^T e^{\beta s}\mid Z'(s)+g(s,X'_s,Y'_s)-\overline{Z}(s)) \mid^2 ds \right]=0.
\end{eqnarray*}
Finally we get $X'=\overline{X},\; Y'=\overline{Y}$ and $Z'=\overline{Z}-g(.,\overline{X},\overline{Y})$, which and the proof of Theorem \ref{T1}.

\begin{rem}
In this work we have established a result of existence and uniqueness of FBSDDEs which extends existing work in the scientific literature. First, our paper extends the different results of existence and uniqueness of FBSDEs  without delay, established with several methods respectively in \cite{HP}, \cite{Ma}, etc. Next, our works extend existence and uniqueness result for BSDEs with delayed generator establish by Delong and Imkeller in \cite{DI}. Our condition of existence and uniqueness, $24K e\max(1,T^2)<1$, is equivalent with respect to a constant to the square power of  $T$ with the condition which is established in \cite {DI} which is $8Ke\max(1,T)<1$, but for the more complex type of equation.
\end{rem}
\section{Comments and remarks}

In this section we give a brief overview  on possible applications of introduced  FBSDDEs as a motivation for this paper, as well as some ideas for the future work and extension of the research.

\medskip

Our paper we extended the class of solutions to non homogenous version  of BSDE with delay under particular version of Lipschitz condition for the coefficients, as well as proved the existence and uniqueness result for the FBSDDEs. It would be interesting to explore some versions of non Lipschitz conditions and prove the existence and uniqueness result  and it would be a topic for future work.

Many natural and social phenomena shows that the state process at time $t$ depends not only on its present state but also its past history.  Application of FBSDEs in mathematical finance is already  known where underling asset (or appropriate process of interest) is described via backward part of eq., while the wealth is described via forward part of eq.\ .  In accordance to this, FBSDDEs can be used to model interesting futures from economy and finance where the present value of the point of interest is influenced by the past values, and/or there is a time lag between the realizations and the actual influence/consequences. This can be very much applied to the epidemiology as well as in the mentioned financial and economical application fields, where per example the time between the contact with the virus and manifestation of it differs. Also, the probability of the manifestation is highly dependent of the values from the past. We believe that  FBSDDE if the form which we introduced can catch this properties.  In all mentioned potential applications, there is always  a goal to minimise/maximise value of the process of interest, where this optimisation can be put in stochastic optimal control  (shorter SOC) framework . SOC problems with different conditions on the generators are the topics for future work.

\bigskip

\bigskip

{\bf Acknowledgment:} Jasmina \DJ or\dj evi\'c is supported by RCN Project nr. 274410 --- Stochastics for Time-Space Risk Models (STORM).

\end{document}